\documentclass[a4paper, 12pt]{article}

\usepackage[a4paper]{geometry}
\usepackage[english]{babel}
\usepackage{amsmath, amsfonts, amssymb, amsthm, mathrsfs, mathtools}
\usepackage{nicematrix}
\usepackage{csquotes}
\usepackage{indentfirst}
\usepackage{chngcntr}
\usepackage{enumitem}
\usepackage{titlesec}
\usepackage{caption}
\usepackage{cmap}
\usepackage{floatrow}
\usepackage{algpseudocode}
\usepackage{algorithm}
\usepackage{lmodern}
\usepackage[colorlinks, linkcolor = purple, citecolor = purple]{hyperref}

\mathtoolsset{showonlyrefs=true}
\titleformat{\section}
 {\normalfont\fontsize{16}{16}\bfseries}{\thesection}{1em}{}

\DeclareFontEncoding{LS2}{}{\noaccents@}
\DeclareFontSubstitution{LS2}{stix}{m}{n}
\DeclareSymbolFont{arrows3}{LS2}{stixtt}{m}{n}
\DeclareMathSymbol{\squareulblack}{\mathord}{arrows3}{"88}

\theoremstyle{definition}
\newtheorem{definition}{Definition}
\newtheorem{Proposition}{Proposition}
\newtheorem{lemma}{Lemma}
\newtheorem{theorem}{Theorem}
\newtheorem{problem}{Problem}

\newtheorem*{definition*}{Definition}
\newtheorem*{lemma*}{Lemma}
\newtheorem*{theorem*}{Theorem}
\newtheorem*{corollary*}{Corollary}

\DeclareMathOperator*{\tr}{tr}
\DeclareMathOperator*{\rank}{rank}
\DeclareMathOperator*{\argmin}{argmin}

\renewcommand{\geq}{\geqslant}
\renewcommand{\leq}{\leqslant}

\newcommand{\R}{\mathbb{R}}

\newcommand{\cS}{\mathcal{S}}
\newcommand{\cR}{\mathcal{R}}

\newcommand{\normxi}[1]{\Vert{#1}\Vert_\xi}

\newcommand{\norms}[1]{\Vert{#1}\Vert_2}
\newcommand{\matdec}[1]{\mathrm{#1}}

\newcommand{\eqspace}{\\[1em]}

\begin{document}

\title{Subset selection for matrices in spectral norm}

\author{
Ivan Kozyrev
\thanks{Moscow Institute of Physics and Technology, Dolgoprudny, Moscow Region, 141701, Russian Federation.}
\and
Alexander Osinsky
\thanks{Skolkovo Institute of Science and Technology, Moscow, 121205, Russian Federation; \\ Marchuk Institute of Numerical Mathematics RAS, Moscow, 119333, Russian Federation.}
}

\maketitle

\begin{abstract}
    We address the subset selection problem for matrices, where the goal is to select a subset of $k$ columns from a \enquote{short-and-fat} matrix $X \in \R^{m \times n}$, such that the pseudoinverse of the sampled submatrix has as small spectral or Frobenius norm as possible. For the NP-hard spectral norm variant, we propose a new deterministic approximation algorithm. Our method refines the potential-based framework of spectral sparsification by specializing it to a single barrier function. This key modification enables direct, unweighted column selection, bypassing the intermediate weighting step required by previous approaches. It also allows for a novel adaptive update strategy for the barrier. This approach yields a new, explicit bound on the approximation quality that improves upon existing guarantees in key parameter regimes, without increasing the asymptotic computational complexity. Furthermore, numerical experiments demonstrate that the proposed method consistently outperforms its direct competitors. A complete C++ implementation is provided to support our findings and facilitate future research.\vspace{6pt}
    
    \vspace{6pt}
    
\end{abstract}

\section{Introduction}

\subsection{Subset selection for matrices}

Given a short-and-fat matrix $X \in \R^{m \times n}$ (i.e., $m < n$, and often $m \ll n$ in typical applications), the problem of selecting a subset of its columns that \enquote{optimally represents} the original matrix $X$ is often of interest. A common optimality criterion, arising in diverse applied areas, is to minimize the norm of the Moore-Penrose pseudoinverse of the submatrix formed by the selected columns. This objective leads to the following combinatorial problem:

\begin{problem}[Subset selection for matrices]\label{prm:subset_selection}
Given a full-rank matrix $X \in \R^{m \times n}$ with $m < n$ and a sampling parameter $k \in \overline{m,n}$ (where $\overline{a,b}$ denotes the set of integers $\{a, a+1, \dots, b\}$), find a set of column indices $\cS_{opt} \subseteq \overline{1, n}$ such that $|\cS_{opt}| \leq k$, $\rank(X_{\cS_{opt}}) = m$, and $\normxi{X_{\cS_{opt}}^\dag}$ is minimized. Formally,
\[
\cS_{opt} \in \argmin_{\cS \in \mathcal{F}(X, k)} \normxi{X_\cS^\dag}\,,
\]
where $\mathcal{F}(X, k) =\left\{\cS \subseteq \overline{1,n} \,:\, |\cS| \leq k \text{ and } \rank \left( X_\cS \right) = m \right\}$, and $\xi \in \{2, F\}$ denotes the spectral or Frobenius norm, respectively.
\end{problem}

Here, $X_\cS$ denotes the submatrix of $X$ containing the columns indexed by $\cS$, and $X_\cS^\dag$ is its Moore-Penrose pseudoinverse.

A brute-force approach to solving Problem~\ref{prm:subset_selection} involves evaluating $\normxi{X_\cS^\dag}$ for all $\cS \in \mathcal{F}(X, k)$. However, this is computationally infeasible for matrices of practical dimensions. Furthermore, the spectral norm version ($\xi=2$) of Problem~\ref{prm:subset_selection} is NP-hard, as shown by Çivril and Magdon-Ismail~\cite{subset_selection_complexity}. A similar NP-hardness result exists for the Frobenius norm case ($\xi=F$) when the sampling parameter $k=m$~\cite{doi:10.1287/moor.2021.1129}. These computational hardness results motivate the development of efficient approximation algorithms and heuristics.

\subsection{Applications}

Problem~\ref{prm:subset_selection} arises in numerous research fields. In statistics and machine learning, it is fundamental to optimal experimental design~\cite{Huan_Jagalur_Marzouk_2024, doi:10.1287/moor.2021.1129, allen2021near} and feature selection for tasks like k-means clustering~\cite{6488848, boutsidis2014randomized}. In graph theory, it corresponds to finding low-stretch spanning trees~\cite{faster_subset_selection} and is related to the algebraic connectivity of graphs~\cite{Fiedler1973, lamperskisimple}. Within numerical linear algebra, it underpins methods for sparse least-squares regression~\cite{BOUTSIDIS2014273}, rank-deficient least squares problems~\cite{10.1145/1132973.1132981, doi:10.1137/090780882}, and preconditioning~\cite{Arioli2015-eb}. Further applications are found in graph signal processing~\cite{chen2015discrete, tsitsvero2016signals} and multipoint boundary value problems~\cite{DEHOOG2007349, DEHOOG20111845}. Generalized versions of Problem~\ref{prm:subset_selection} broaden its applicability even further; see, for instance,~\cite{brown2024maximizingminimumeigenvalueconstant, lamperskisimple}.

Another key application area, and the primary motivation for the present work, lies in low-rank matrix approximation, specifically within the theory of $\matdec{CW}$ (column-based)~\cite{column_based_reconstruction, Osinsky2023-bg} and $\matdec{CUR}$ (cross)~\cite{doi:10.1137/140977898, GOREINOV19971} approximations. 

The achievable accuracy of these approximations relative to the optimal truncated $\matdec{SVD}$ depends directly on the ability to select rows from the leading $r$ singular vectors such that the resulting submatrix has a pseudoinverse with a small norm. This connection has been examined for pseudo-skeleton $\matdec{CUR}$ approximations in~\cite{OSINSKY2018221, OSINSKY2025} and for $\matdec{CW}$ approximations in~\cite{Osinsky2023-bg}. This task is a transposed version of Problem~\ref{prm:subset_selection}, where $m = r$, $k \geq r$ is the number of selected rows/columns for the approximation, and $n$ is the number of rows/columns of the initial matrix (typically much larger than $k$).

From a theoretical perspective, this implies that the relative accuracy of $\matdec{CW}$ and $\matdec{CUR}$ approximations depends on the quantities $t_\xi(m, k, n)$. These quantities capture the worst-case scenario (in terms of the minimum achievable norm of the pseudoinverse) for Problem~\ref{prm:subset_selection} over matrices with orthonormal rows:
\begin{equation}\label{eqn:t_m_k_n}
t_\xi(m, k, n) = \adjustlimits\sup_{X \in \mathcal{O}(m, n)} \min_{\cS \in \mathcal{F}(X, k)} \normxi{X_\cS^\dag}\,,\quad \xi \in \{2, F\}\,,
\end{equation}
where $\mathcal{O}(m, n)$ denotes the set of $m \times n$ matrices with orthonormal rows. The exact values of $t_\xi(m, k, n)$ remain unknown; current best upper bounds are derived from theoretical guarantees provided by approximation algorithms for Problem~\ref{prm:subset_selection}.

In practice, if an algorithm for Problem~\ref{prm:subset_selection} yields a submatrix $X_\cS$ satisfying a bound of the form $\normxi{X_\cS^\dag}^2 \leq f_\xi(m, k, n) \normxi{X^\dag}^2$, then low-rank approximations can be constructed with accuracy guarantees where $t_\xi^2(m, k, n)$ is effectively replaced by $f_\xi(m, k, n)$. Strictly speaking, achieving the theoretically best possible bound requires knowledge of the (truncated) singular value decomposition of the original matrix, which can be computationally expensive. Nevertheless, advances in solving Problem~\ref{prm:subset_selection} directly translate to improved low-rank approximation techniques.

\subsection{Our contributions}

This paper makes the following primary contributions:
\begin{enumerate}
    \item We develop an approximation algorithm (Algorithm~\ref{alg:main}) for the spectral norm variant of Problem~\ref{prm:subset_selection}. Our proposed algorithm is deterministic and greedy, drawing inspiration from Algorithm~3 in~\cite{faster_subset_selection}. It achieves a new, explicit bound on the norm of the pseudoinverse that improves upon existing results, while maintaining a computational complexity of $O(n k m^2)$ for dense matrices.
    
    \item We demonstrate the practical effectiveness of Algorithm~\ref{alg:main} through numerical experiments, showing it consistently outperforms other tested algorithms, including its direct predecessor, Algorithm~3 in~\cite{faster_subset_selection}.
\end{enumerate}

The key properties of Algorithm~\ref{alg:main} are formally stated in Theorem~\ref{thm:main}. For a selected submatrix $X_\cS$, the algorithm guarantees the following bound:
\begin{equation}\label{eqn:tighter_bound}
    \normxi{X_\cS^\dag}^2 \leq \frac{n}{m}\left(\frac{\sqrt{(k - 1)m + 1} - 1}{\sqrt{(k - 1)m + 1} - k}\right)^2 \normxi{X^\dag}^2\,, \quad \xi \in \{2,F\}\,.
\end{equation}
This expression has a removable singularity at $m = k = 1$; in this specific case, the bound is understood as the limit as $k \to 1+$, yielding $\normxi{X_\cS^\dag}^2 \leq n \normxi{X^\dag}^2$. 

The computational complexity of Algorithm~\ref{alg:main} is $O(k m^3 + k m T_X)$, where $T_X$ is the cost of multiplying $X^T$ by an $m$-dimensional vector. The $k m T_X$ term stems from the algorithm's $k$ iterations; the dominant cost in each is a matrix product equivalent to $m$ such vector multiplications. For a general dense matrix where $T_X=O(nm)$, the total complexity becomes $O(n k m^2)$.

While Equation~\ref{eqn:tighter_bound} provides a tight bound, the following slightly looser but more interpretable bound also holds:
\begin{equation}\label{eqn:looser_bound}
    \normxi{X_\cS^\dag}^2 \leq \frac{n}{\left(\sqrt{k} - \sqrt{m - 1}\right)^2} \normxi{X^\dag}^2\,, \quad \xi \in \{2,F\}\,.
\end{equation}
Notably, for $m=1$, both bounds~\eqref{eqn:tighter_bound} and~\eqref{eqn:looser_bound} are equivalent and optimal.

\subsection{Related Work}\label{sec:related_work}

This section reviews work relevant to Problem~\ref{prm:subset_selection}, organized into two main areas: approximation algorithms and lower bounds.

\subsubsection{Approximation algorithms}\label{sec:approximation_algorithms}

Numerous deterministic and randomized algorithms have been proposed to find approximate solutions to Problem~\ref{prm:subset_selection}; a comprehensive survey can be found in~\cite{faster_subset_selection}. More recent contributions include Algorithm~1 in~\cite{subset_selection_with_fixed_blocks} and Algorithm~3 in~\cite{Osinsky23}. Table~\ref{tbl:subset_selection_algorithms} lists deterministic algorithms constituting direct competitors to our Algorithm~\ref{alg:main}, by virtue of being specifically designed for the spectral norm variant of Problem~\ref{prm:subset_selection} and representing relevant existing approaches.

\begin{table}[htbp] 
\begin{NiceTabular*}{\textwidth}{@{\extracolsep{\fill}} |l|c|c|c|} 
\hline
Algorithm & Sampling & Upper bound on $\frac{\norms{X_\cS^\dag}^2}{\norms{X^\dag}^2}$ & Operation Count \\ 
\hline \hline
\Block{2-1}{Alg.~2 in~\cite{faster_subset_selection}} & \Block{2-1}{$k \geq m$} & \Block{2-1}{$1 + \frac{m(n - k)}{k - m + 1}$} & \Block{2-1}{$O(nm^2 + nm(n - k))$} \\
& & & \\
\hline
\Block{2-1}{Alg.~3 in~\cite{faster_subset_selection}} & \Block{2-1}{$k > m$} & \Block{2-1}{$\left(\frac{\sqrt{n} + \sqrt{k}}{\sqrt{k} - \sqrt{m}}\right)^2$} & \Block{2-1}{$O(nkm^2)$} \\
& & & \\
\hline
\Block{2-1}{Alg.~1 in~\cite{subset_selection_with_fixed_blocks}} & \Block{2-1}{$k \geq m$} & \Block{2-1}{$\frac{n^2}{\left( \sqrt{(k + 1)(n - m)} - \sqrt{m(n - k - 1)} \right)^2 }$} & \Block{2-1}{$O(nk m^{\theta} )$} \\\
& & & \\
\hline
\end{NiceTabular*}
\caption{Summary of deterministic approximation algorithms for the spectral norm version of Problem~\ref{prm:subset_selection}. Here, $\theta > 2$ is the matrix multiplication exponent; we assume a fixed target precision, thus logarithmic factors in precision are omitted.}
\label{tbl:subset_selection_algorithms}
\end{table}

A direct comparison of the bounds in Table~\ref{tbl:subset_selection_algorithms} with the looser bound~\eqref{eqn:looser_bound} for our Algorithm~\ref{alg:main} reveals the following regimes:
\begin{itemize}
    \item For $k \leq m + 3$, Algorithm~2 in~\cite{faster_subset_selection} provides a tighter bound.
    \item For $m + 3 < k \leq n/m - 1$ (assuming this interval is non-empty), our Algorithm~\ref{alg:main} exhibits the best theoretical bound among these methods.
    \item For larger $k$ (specifically $k > n/m - 1$), the bound of Algorithm~1 in~\cite{subset_selection_with_fixed_blocks} becomes tighter.
\end{itemize}

In terms of complexity, Algorithm~\ref{alg:main} and Algorithm~3 in~\cite{faster_subset_selection} are identical at $O(nkm^2)$. Algorithm~2 in~\cite{faster_subset_selection} is efficient only for large values of $k$, as its complexity can otherwise reach $O(n^2 m)$. The complexity of Algorithm~1 in~\cite{subset_selection_with_fixed_blocks} depends on the matrix multiplication exponent $\theta$. Assuming classical matrix multiplication ($\theta = 3$, even the simplest Strassen method does not bring any improvements in practice until $m$ reaches at least 500~\cite{fast_shtrassen}), our algorithm is faster by a factor of $O(m)$, and the constant is also much lower due to the fact that the main cost in Algorithm~1 in~\cite{subset_selection_with_fixed_blocks} comes from computing roots of the characteristic polynomials and is dependent on the numerical precision, which must be high enough to calculate the roots accurately.

\subsubsection{Lower bounds}

Many approximation algorithms for Problem~\ref{prm:subset_selection}, including our Algorithm~\ref{alg:main}, provide guarantees of the form $\normxi{X_\cS^\dag}^2 \leq f_\xi(m, k, n) \normxi{X^\dag}^2$. A natural question arises: How close are these bounds to the theoretical optimum? This leads to the consideration of lower bounds.

\begin{definition}
    A non-negative number $\gamma$ is called a \textit{lower bound} for given $m \leq k \leq n$ if
    \[
    \gamma \leq \adjustlimits\sup_{X \in GL(m, n)} \min_{\cS \in \mathcal{F}(X, k)} \frac{\normxi{X_\cS^\dag}^2}{\normxi{X^\dag}^2}\,,\quad \xi \in \{2, F\}\,,
    \]
    where $GL(m, n)$ denotes the set of full-rank matrices of size $m \times n$. This concept is analogous to $t_\xi$ introduced in~\eqref{eqn:t_m_k_n}, but here $X$ can be arbitrary full-rank matrix, which leads to the different behaviour for $\xi = F$. 
\end{definition}

Several lower bounds for both spectral and Frobenius norms have been derived in~\cite{faster_subset_selection} and improved in \cite{Osinsky23} and \cite{Osinsky2023-bg}. Specifically, for both $\xi = 2$ and $\xi = F$ the current best lower bound is $\gamma = (n-m+1)/(k-m+1)$ (Proposition 1 in~\cite{Osinsky23}). For the Frobenius norm case ($\xi = F$), this bound is tight, as it is achieved by Algorithm~1 in~\cite{faster_subset_selection}.

Examining the bound for Algorithm~\ref{alg:main} \eqref{eqn:tighter_bound}, which is the same for spectral and Frobenius norms, we obtain that for fixed $n/k$ and $n \to \infty$, it is 
\[
\frac{n}{k} \left( 1 + O \left( \sqrt{\frac{m}{k}} \right) \right)\,,
\]
asymptotically matching the lower bounds mentioned earlier.

\subsection{Organization}

The remainder of this paper is organized as follows. Section~\ref{sec:algorithm} develops our proposed algorithm and its theoretical underpinnings, culminating in the proof of our main result, Theorem~\ref{thm:main}. We conclude this section by situating our method within the spectral sparsification framework and highlighting its key distinctions from prior work. Section~\ref{sec:numerical_experiments} then presents numerical experiments comparing the performance of Algorithm~\ref{alg:main} against existing methods.

\section{A deterministic greedy selection algorithm}\label{sec:algorithm}

While the core idea of our algorithm is related to Algorithm~3 from~\cite{faster_subset_selection} and ultimately to~\cite{twice_ramanujan_sparcifiers}, our framework incorporates significant modifications. We therefore present a self-contained proof of all results.

\subsection{The goal of the greedy selection process}

Let $\cS \subseteq \overline{1, n}$ be the set of currently selected column indices, and $\cR = \overline{1,n} \setminus \cS$ be the set of remaining (candidate) indices. To establish a greedy selection process, we must define a selection criterion that guides the choice of columns at each step. Directly minimizing $\norms{X_\cS^\dag}^2$, however, is a poor choice for two reasons:
\begin{itemize}
    \item During the selection process, matrix $X_\cS$ will have rank less than $m$. When the rank jumps from $r$ to $r + 1$, the smallest non-zero singular value becomes very small, causing $\norms{X_\cS^\dag}^2$ to explode. A greedy strategy minimizing this norm would therefore avoid selecting columns that increase the rank, which is antithetical to the goal of finding a full-rank well-conditioned submatrix.
    
    \item A greedy approach requires evaluating the objective for each candidate column at every step. Calculating the spectral norm of the pseudoinverse is computationally expensive, with a complexity that is cubic in the matrix dimensions.
\end{itemize}

Therefore, we require an alternative objective that ensures a small value of $\norms{X_\cS^\dag}^2$ when the final subset of size $k$ is formed. Consider the matrix $Y$ built from the currently selected columns $\cS$:
\[
Y = X_\cS X_\cS^T = \sum_{j \in \cS} x_j x_j^T\,.
\]
Once $X_\cS$ has rank $m$, its squared spectral pseudoinverse norm is given by $1/\lambda_m(Y)$, where $\lambda_m(Y)$ is the smallest eigenvalue of $Y$. Our goal is therefore equivalent to maximizing this smallest eigenvalue. Following~\cite{twice_ramanujan_sparcifiers}, we approach this by employing a barrier function.

\begin{definition}\label{def:barrier_func}
    For a symmetric matrix $Y \in \R^{m \times m}$ and a scalar $l < \lambda_m(Y)$,
    \[
    \displaystyle \Phi_l(Y) = \tr (Y - lI)^{-1} = \sum_{j=1}^m \frac{1}{\lambda_j(Y) - l}
    \]
    is a \textit{barrier function} (or \textit{potential}), which measures how \enquote{far} the eigenvalues of $Y$ are from the barrier $l$. The potential $\Phi_l(Y)$ is well-defined only when $l < \lambda_m(Y)$, a condition we will maintain through our analysis.
\end{definition}

This function offers several advantages:
\begin{itemize}
    \item Its value provides a lower bound on $\lambda_m(Y)$ and thus an upper bound on the final value of $\norms{X_\cS^\dag}^2 = \lambda_m^{-1}(Y)$ (when $|\cS|=k$ and $Y=X_\cS X_\cS^T$), since
    \begin{equation}\label{eqn:l_bound}
        \frac{1}{\lambda_m(Y) - l} \leq \Phi_{l}(Y) \implies \lambda_m(Y) \geq l + \frac{1}{\Phi_{l}(Y)}\,.
    \end{equation}
    \item It is monotonically increasing with respect to $l$.
    \item For a fixed barrier level $l$ and any column $x$, $\Phi_{l}(Y + xx^T) \leq \Phi_{l}(Y)$. This follows from Weyl's inequality, which guarantees that the eigenvalues of $Y$ do not decrease when a positive semidefinite matrix is added.
    \item Its value can be efficiently updated when a new column is added to $\cS$, as will be shown in~\eqref{eqn:Phi_recalculation}.
\end{itemize}

This leads to our overall strategy, which is guided by the lower bound in \eqref{eqn:l_bound}. At each step, we seek to advance the barrier $l$, while keeping the potential $\Phi_l(Y)$ bounded. The selection of the next column is based on evaluating the potential $\Phi_{l'}(Y + x_jx_j^T)$ for a special, precalculated $l'$ and all candidate columns $x_j$. As we will show, this evaluation can be performed efficiently, forming the core of our greedy algorithm.

\subsection{One step of the greedy selection process}

Here, we detail the core mechanism of our greedy algorithm: the selection of a single new column. The goal is to find a column that allows us to advance the barrier $l$ significantly, while ensuring that the potential $\Phi_l(Y)$ remains controlled. To formalize this, we seek a guaranteed advancement of the barrier, $\delta$, such that selecting a suitable column ensures the potential does not increase. The analysis is performed for a given potential value, which we denote by $\varepsilon = \Phi_l(Y)$.

The following lemma is central to our analysis. It establishes an inequality that determines $\delta$, forming the basis for our column selection strategy and the subsequent theoretical guarantees. The lemma operates under the assumption that $X$ has orthonormal rows (i.e., $XX^T = I$); this constraint will be relaxed later.

\begin{lemma}\label{lem:delta_inequality}
    Fix $X \in \R^{m \times n}$ ($XX^T = I$) and set $\cS \subseteq \overline{1, n}$ of cardinality $i < n$. Let $Y = X_\cS X_\cS^T$, $l < \lambda_m(Y)$ and $\Phi_l(Y) = \varepsilon$. If $\delta < \varepsilon^{-1}$ satisfies
    \begin{equation}\label{eqn:delta_inequality}
        \frac{1 - l - m/\varepsilon}{n - i}(1 - \varepsilon \delta) \geq \delta \left( 1 - \frac{\varepsilon \delta}{m}\right)\,,
    \end{equation}
    then $\lambda_m(Y) > l + \delta$, and there exists $j \in \cR = \overline{1,n} \setminus \cS$, such that $\Phi_{l + \delta}(Y + x_j x_j^T) \leq \Phi_l(Y) = \varepsilon$.
\end{lemma}
 
\begin{proof}
We define a column $x_j$ (for $j \in \cR$) as \enquote{good} if, for the new barrier $l' = l + \delta$, it satisfies $\Phi_{l'}(Y + x_jx_j^T) \leq \Phi_l(Y) = \varepsilon$. This proof derives a sufficient condition on $\delta$ to guarantee the existence of at least one such \enquote{good} column.

\begin{enumerate}[wide]
    \item First, we restrict $\delta < \varepsilon^{-1}$. This ensures that the new barrier $l' = l + \delta$ remains below the smallest eigenvalue of $Y$:
    \[
    \frac{1}{\lambda_m(Y) - l} \leq \Phi_l(Y) = \varepsilon \implies \lambda_m(Y) - l \geq \frac{1}{\varepsilon}\,.
    \]
    Since $\delta < \varepsilon^{-1}$, we have $\lambda_m(Y) > l + \delta$. This proves the first auxiliary claim of the lemma and ensures that potentials such as $\Phi_{l'}(Y)$ and $\Phi_{l'}(Y + xx^T)$ are well-defined, allowing us to proceed with the main proof.
    
    \item Now, for an arbitrary column $x$
    \begin{equation}\label{eqn:Phi_recalculation}
    \Phi_{l'}\left(Y + xx^T\right) = \tr \left( Y - l'I + xx^T \right)^{-1} \overset{(*)}{=} \Phi_{l'}(Y) - \frac{x^T(Y-l'I)^{-2}x}{1 + x^T(Y - l'I)^{-1}x}\,,  
    \end{equation}
    $(*)$ follows from the Sherman-Morrison formula and the cyclic property of the trace. Thus, for $\delta > 0$ we can write down
    \[
    \Phi_{l'}\left(Y + xx^T\right) \leq \Phi_l(Y) \Longleftrightarrow 1 + x^T(Y - l'I)^{-1}x - \frac{x^T(Y-l'I)^{-2}x}{\Phi_{l'}(Y) - \Phi_l(Y)} \leq 0\,.
    \]

    \item To guarantee that at least one \enquote{good} column exists, we employ the following averaging argument:
    \[
    \sum_{j \in \cR}\left( 1 + x_j^T(Y - l'I)^{-1}x_j - \frac{x_j^T(Y-l'I)^{-2}x_j}{\Phi_{l'}(Y) - \Phi_l(Y)} \right) \leq 0\,.
    \]
    Using the cyclic property and linearity of the trace, this is equivalent to
    \[
    |\cR| + \tr\left((Y - l'I)^{-1}\sum_{j \in \cR}x_jx_j^T\right) - \frac{\tr\left((Y - l'I)^{-2}\sum_{j \in \cR}x_jx_j^T\right)}{\Phi_{l'}(Y) - \Phi_l(Y)} \leq 0\,.
    \]
    Since $X$ has orthonormal rows, $\sum_{j \in \cR}x_jx_j^T = I - Y$, and condition becomes
    \[
    \frac{\tr\left((Y - l'I)^{-2}\left[I - Y\right]\right)}{\Phi_{l'}(Y) - \Phi_l(Y)} - \tr\left((Y-l'I)^{-1}\left[I - Y\right]\right) \geq n - i\,.
    \]
    Using the \enquote{clever zero} trick, $I - Y = (1 - l')I - (Y - l'I)$, we obtain
    \[
    (1-l')\left(\frac{\tr(Y - l'I)^{-2}}{\Phi_{l'}(Y) - \Phi_l(Y)} - \Phi_{l'}(Y)\right) + \left( m - \frac{\Phi_{l'}(Y)}{\Phi_{l'}(Y) - \Phi_l(Y)}\right) \geq n - i\,. 
    \]

    \item To simplify this condition, we bound its constituent terms. Let $\lambda_j \equiv \lambda_j(Y)$. The expression in the second pair of brackets can be bounded as follows:
    \begin{align*}
         m - \frac{\Phi_{l'}(Y)}{\Phi_{l'}(Y) - \Phi_l(Y)} &= m - \frac{\Phi_{l'}(Y)}{\sum_{j = 1}^m \left[ (\lambda_j - l')^{-1} - (\lambda_j - l)^{-1} \right]} \\
         &= m - \frac{\Phi_{l'}(Y)}{\delta \sum_{j = 1}^m (\lambda_j - l')^{-1}\cdot(\lambda_j - l)^{-1}}\\ &\overset{(*)}{\geq} m - \frac{m\Phi_{l'}(Y)}{\delta \Phi_l(Y)\Phi_{l'}(Y)} = -\frac{m}{\varepsilon}\left( \frac{1}{\delta} - \varepsilon \right)\,,
    \end{align*}
    $(*)$ follows from applying Chebyshev's sum inequality to denominator.

    A bound for the expression in the first pair of brackets is more involved and is deferred to Proposition~\ref{stt:first_braces_inequality}. Applying these two bounds yields the sufficient condition: 
    \[
    \left( 1 - l - \frac{m}{\varepsilon} \right) \left(\frac{1}{\delta} - \varepsilon\right) \geq (n - i) \left( 1 - \frac{\varepsilon \delta}{m}\right)\,.
    \]

    \item For $\delta = 0$ every column indexed in $\cR$ is \enquote{good}, and to include that case we can multiply inequality by $\delta$. The formulation in~\eqref{eqn:delta_inequality} correctly includes this trivial case.\qedhere
\end{enumerate}
\end{proof}

\begin{Proposition}[Stronger version of Claim 3.6 from~\cite{twice_ramanujan_sparcifiers}]
\label{stt:first_braces_inequality} 
For symmetric matrix $Y \in \R^{m \times m}$ and $l' = l + \delta$ where $\delta > 0$ and $l' < \lambda_m(Y)$,
\begin{equation}\label{eq:prop1}
\frac{\tr(Y - l'I)^{-2}}{\Phi_{l'}(Y) - \Phi_l(Y)} - \Phi_{l'}(Y) \geq \frac{1/\delta - \varepsilon}{1 - \varepsilon \delta / m}\,.
\end{equation}
\end{Proposition}

\begin{proof}
    \begin{enumerate}[wide]
    \item We start by deriving a lower bound for $\Phi_{l'}(Y)$. By writing $\Phi_{l'}(Y) = \Phi_{l}(Y) + (\Phi_{l'}(Y) - \Phi_{l}(Y))$ and expanding the difference term, we have:
    \[
    \Phi_{l'}(Y) = \varepsilon + \delta\sum_{j = 1}^m (\lambda_j - l')^{-1}\cdot(\lambda_j - l)^{-1} \overset{(*)}{\geq} \varepsilon + \frac{\varepsilon\delta \Phi_{l'}(Y)}{m}\,,
    \]
    $(*)$ follows from Chebyshev's sum inequality. The above implies 
    \begin{equation}\label{eqn:phi_l_new_bound}
        \Phi_{l'}(Y) \geq \frac{\varepsilon}{1 - \varepsilon \delta / m}\,.
    \end{equation}
    
\item Now we return to Equation~\ref{eq:prop1}. Letting $\lambda_j \equiv \lambda_j(Y)$, we analyse the terms on the left-hand side:
    \begin{gather*}
    \frac{\tr(Y - l'I)^{-2}}{\Phi_{l'}(Y) - \Phi_l(Y)} = \frac{\sum_{j = 1}^m (\lambda_j - l')^{-2}}{\delta\sum_{j = 1}^m (\lambda_j - l')^{-1}(\lambda_j - l)^{-1}} = \frac{1}{\delta} + \frac{\sum_{j = 1}^m (\lambda_j - l')^{-2}(\lambda_j - l)^{-1}}{\delta\sum_{j = 1}^m (\lambda_j - l')^{-1}(\lambda_j - l)^{-1}}\,, \eqspace
    \Phi_{l'}(Y) = \varepsilon + \delta\sum_{j = 1}^m (\lambda_j - l')^{-1}(\lambda_j - l)^{-1} \overset{(*)}{\leq} \varepsilon + \varepsilon\delta\frac{\sum_{j = 1}^m (\lambda_j - l')^{-2}(\lambda_j - l)^{-1}}{\sum_{j = 1}^m (\lambda_j - l')^{-1}(\lambda_j - l)^{-1}}\,,
    \end{gather*}
    $(*)$ follows from Cauchy–Schwarz inequality for $a_j = (\lambda_j -l')^{-1},b_j = (\lambda_j - l)^{-1}$ in the form $(\sum_{j=1}^m a_j b_j)^2 \leq (\sum_{j=1}^m a_j^2 b_j ) (\sum_{j=1}^m b_j)$. Thus for the whole expression:
    \begin{align*}
    \frac{\tr(Y - l'I)^{-2}}{\Phi_{l'}(Y) - \Phi_l(Y)} - \Phi_{l'}(Y) &\geq 
    \frac{1}{\delta} - \varepsilon + (1 - \varepsilon \delta)\frac{\sum_{j = 1}^m (\lambda_j - l')^{-1} \cdot (\lambda_j - l')^{-1}(\lambda_j - l)^{-1}}{\sum_{j = 1}^m (\lambda_j - l')^{-1}(\lambda_j - l)^{-1}}\\
    &\overset{(a)}{\geq} \left( \frac{1}{\delta} - \varepsilon \right) \left( 1 + \frac{\delta \Phi_{l'}(Y)}{m} \right) \overset{(b)}{\geq} \frac{1/\delta - \varepsilon}{1 - \varepsilon\delta/m}\,,
    \end{align*}
    $(a)$ follows from applying Chebyshev's sum inequality to numerator, $(b)$ follows from using inequality~\eqref{eqn:phi_l_new_bound}.\qedhere
    \end{enumerate}
\end{proof}

The inequality~\eqref{eqn:delta_inequality} from Lemma~\ref{lem:delta_inequality} is quadratic in $\delta$. To characterize its solution set, we must prove that the corresponding equation has real roots and determine their relation to the critical value $\varepsilon^{-1}$. The following proposition addresses these points.

\begin{Proposition}\label{stt:properties_of_delta_equation}
Let $m, n, i\in \mathbb{N}$, $i < n$, and $\varepsilon \in \mathcal{I}_m$, $ l \in [-m/\varepsilon, 1- m/\varepsilon]$, where
\begin{equation}\label{eqn:I_m}
    \mathcal{I}_m = \begin{cases}
    (0, 1) & \text{if }\,m=1,\\
    (0, \infty) & \text{otherwise.}
    \end{cases}
\end{equation}
Then, the following quadratic equation
\begin{equation}\label{eqn:delta_equality}
\frac{1 - l - m/\varepsilon}{n - i}(1 - \varepsilon \delta) = \delta\left(1 - \frac{\varepsilon \delta}{m}\right)
\end{equation} 
has two real, non-negative roots $\delta^*$ and $\delta^{**}$ satisfying $\delta^* < \varepsilon^{-1} \leq \delta^{**}$ and $1 - (l + \delta^* + m/\varepsilon) \geq 0$.
\end{Proposition}

\begin{proof}\begin{enumerate}[wide]
    \item Under the proposition's assumption that $l \leq 1 - m/\varepsilon$, we have $f(0) \geq g(0)$. Furthermore, at $\delta = \varepsilon^{-1}$, we have $f(\varepsilon^{-1}) = 0$ and $g(\varepsilon^{-1}) = \varepsilon^{-1}(1-1/m) \geq 0$, which implies $f(\varepsilon^{-1}) \leq g(\varepsilon^{-1})$. Since $f(\delta)$ is a line with negative slope and $g(\delta)$ is a concave parabola, Equation~\ref{eqn:delta_equality} has two real roots satisfying $0 \leq \delta^* \leq \varepsilon^{-1} \leq \delta^{**}$.

    \item To show that $\delta^* < \varepsilon^{-1}$, we consider cases $m = 1$ and $m > 1$ separately. First, let $m = 1$, then the roots of the Equation~\ref{eqn:delta_equality} are given by
    \[
    \delta^* = \frac{1 - l - m/\varepsilon}{n - i}\,,\ \delta^{**} = \varepsilon^{-1}\,,
    \]
    where $\delta^* \leq 1$ due to the restrictions on $l$ and $n - i$, while $\delta^{**} = \varepsilon^{-1} > 1$ due to the restrictions on $\varepsilon$. Now, let $m > 1$. Then $f(\varepsilon^{-1}) < g(\varepsilon^{-1})$ which implies $\delta^* < \varepsilon^{-1}$.

    \item To prove the final property, it is convenient to rewrite \eqref{eqn:delta_equality} by isolating $\delta$. We introduce the function $\Gamma$:
    \begin{equation}\label{eqn:gamma_def}
    \Gamma_{m, \varepsilon}(\delta) = \delta \frac{1 - \varepsilon \delta / m}{1 - \varepsilon \delta}\,,
    \end{equation}
    defined for $\delta \in [0, \varepsilon^{-1})$. It is easy to verify that $\Gamma_{m, \varepsilon}(\delta)$ and $\Gamma_{m, \varepsilon}(\delta) - \delta$ are non-decreasing functions and $\Gamma_{m, \varepsilon}(\delta) \geq \delta$. Since $\delta^*$ is the only root in right-open interval $[0, \varepsilon)$, it is always given by
    \begin{equation}\label{eqn:gamma_delta*}
    \Gamma_{m, \varepsilon}(\delta^*) = \frac{1 - l - m/\varepsilon}{n - i}\,,
    \end{equation}
    which implies $\delta^* \leq 1 - l - m/\varepsilon$, since $n - i \geq 1$.\qedhere
\end{enumerate} 
\end{proof}

These results allow us to formalize one iteration of the greedy selection algorithm. The procedure is detailed in Algorithm~\ref{alg:one_greedy_iteration}. Note that we intentionally leave the choice of the new barrier $l_{i + 1}$, in State~\ref{state:updating_l} flexible for now. We will return in Subsections~\ref{sec:final_l_estimation} and~\ref{sec:l_epsilon_heuristic} to define a specific update rule optimized for our theoretical and practical goals.

\begin{algorithm}[htbp]
\caption{Template for one step of the greedy selection algorithm.}\label{alg:one_greedy_iteration}
\begin{algorithmic}[1]
\Require \parbox[t]{\dimexpr\linewidth-\algorithmicindent}{
$X \in \R^{m \times n}$ ($XX^T = I$), set $\cS \subseteq \overline{1, n}$ with cardinality $i < n$, $Y_i = X_\cS X_\cS^T$, \newline and barrier value $l_i < \lambda_m(Y_i)$ such that $\Phi_{l_i}(Y_i) = \varepsilon_i \in \mathcal{I}_m$.}
\Ensure Updated $\cS$, $Y_{i + 1}$, and new values of the barrier and potential: $l_{i + 1}$, $\varepsilon_{i + 1}$.

\State\label{state:calc_delta} $\displaystyle \delta_i \gets$ smaller root of~\eqref{eqn:delta_equality} with $l = l_i$ and $\varepsilon = \varepsilon_i$.
\State\label{state:choose_s} \parbox[t]{\dimexpr\linewidth-\algorithmicindent}{%
For each $j \in \cR = \overline{1, n} \setminus \cS$, compute $\Phi_{l_i + \delta_i}(Y_i + x_jx_j^T)$ using~\eqref{eqn:Phi_recalculation} and choose the index~$s$ that minimizes the potential:
}
\[
s \gets \argmin_{j \in \cR}\ \Phi_{l_i + \delta_i}(Y_i + x_j x_j^T)\,.
\]
\State\label{state:update_vars} $\cS \gets \cS \cup \{ s \}$, $Y_{i + 1} \gets Y_i + x_sx_s^T$.
\State\label{state:updating_l} Choose any $l_{i + 1} < \lambda_m(Y_{i + 1})$ such that $\Phi_{l_{i + 1}}(Y_{i + 1}) = \varepsilon_{i + 1} \in \mathcal{I}_{m}$.
\State \textbf{return} $\cS$, $Y_{i + 1}$, $l_{i + 1}$, $\varepsilon_{i + 1}$
\end{algorithmic}
\end{algorithm}

Note that this algorithm is well-defined and allows for consecutive execution as long as $i < k$. By Proposition~\ref{stt:properties_of_delta_equation}, $\delta_i$ exists and satisfies the conditions of Lemma~\ref{lem:delta_inequality}. Thus, calculating $\Phi_{l_i + \delta_i}(Y_i + x_jx_j^T)$ for arbitrary $j \in \cR$ is permissible in State~\ref{state:choose_s}. The input requirements of the next iteration are then ensured by State~\ref{state:updating_l}.

\subsection{A principled update strategy for the barrier}\label{sec:final_l_estimation}

This subsection specifies a principled method for choosing $l_{i + 1}$ and $\varepsilon_{i + 1}$ in State~\ref{state:updating_l} of Algorithm~\ref{alg:one_greedy_iteration} in order to derive the tightest possible theoretical bounds for our framework. This strategy is outlined in Algorithm~\ref{alg:l_epsilon_theory}.

\begin{algorithm}[htbp]
\caption{A principled strategy for updating $l$ and $\varepsilon$.}\label{alg:l_epsilon_theory}
\begin{algorithmic}[1]
\Require \parbox[t]{\dimexpr\linewidth-\algorithmicindent}{
Matrix $Y_{i + 1} \in \R^{m \times m}$, old values of the barrier and potential: $l_i$, $\varepsilon_i$.}
\Ensure New value of barrier and potential: $l_{i + 1}$, $\varepsilon_{i + 1}$.
\State Compute $l_{i + 1}$ by solving $\Phi_{l}(Y_{i + 1}) = \varepsilon_i$ on $[l_i, \lambda_m(Y_{i + 1}))$ using bisection method.
\State $\varepsilon_{i + 1} \gets \varepsilon_i$.
\State \textbf{return} $l_{i + 1}$, $\varepsilon_{i + 1}$.
\end{algorithmic}
\end{algorithm}

This algorithm is well defined: $\Phi_l(Y_{i + 1})$ is a monotonously increasing function of $l$, and because of properties of barrier function, $\Phi_{l_{i}}(Y_{i + 1}) \leq \varepsilon_i$. Thus, using bisection method is permissible. Additionally, $\varepsilon_{i + 1} = \varepsilon_{i} \in \mathcal{I}_m$. Since $\varepsilon$ stays constant in this approach, we will discard the index in this subsection.

Now we focus on deriving guarantees on $l_k$ and resulting norm of the pseudoinverse $\norms{X_\cS^\dag}$, obtained after running Algorithm~\ref{alg:one_greedy_iteration} with State~\ref{state:updating_l} determined by Algorithm~\ref{alg:l_epsilon_theory} (to which we will refer as Algorithm~\ref{alg:one_greedy_iteration}-\ref{alg:l_epsilon_theory}). To this end, we examine two types of finite sequences with increments related to the smaller root of~\eqref{eqn:delta_equality}. We refer to these two types of sequences as epichains and subchains, respectively.

\begin{definition}\label{def:epichain}
    A finite sequence $\{a_i\}_{i=j}^k$ is called an \textit{epichain} for positive integers $j,k,m,n$, where $m \leq k \leq n$, and $\varepsilon \in \mathcal{I}_m$ if:
    \begin{enumerate}
        \item $a_i \leq 1 - m/\varepsilon$ for all $i \in \overline{j, k}$.
        \item $a_{i + 1} - a_i \geq \delta_i$ for all $i \in \overline{j, k-1}$, where $\delta_i$ is the smaller root of \eqref{eqn:delta_equality} with $l=a_i$.
    \end{enumerate}
\end{definition}

\begin{Proposition}\label{stt:l_is_epichain}
    The sequence of the barrier values $\{l_i\}_{i=j}^k$, obtained by running Algorithm~\ref{alg:one_greedy_iteration}-\ref{alg:l_epsilon_theory} for $k - j$ iterations, is an epichain for the given $l,k,m,n$ and the chosen $\varepsilon \in \mathcal{I}_m$. 
\end{Proposition}

\begin{proof}
    Condition 1 follows from the fact that in Algorithm~\ref{alg:one_greedy_iteration}
    \[
    \varepsilon = \Phi_{l_i}(Y_i) = \sum_{j=1}^m \frac{1}{\lambda_j(Y_i) - l_i} \geq \frac{m}{1 - l_i}  \implies l_i \leq 1 - \frac{m}{\varepsilon}\,.
    \] 
    Condition 2 follows from the fact that $\Phi_{l_i + \delta_i}(Y_i) \leq \varepsilon$, which, combined with monotonicity of $\Phi_l(Y_i)$, yields $l_{i + 1} \geq l_i + \delta_i$.\qedhere
\end{proof}

\begin{definition}\label{def:subchain}
    A finite sequence $\{a_i\}_{i=j}^k$ is called a \textit{subchain} for positive integers $j,k,m,n$, where $m \leq k \leq n$, and $\varepsilon \in \mathcal{I}_m$ if $0 \leq a_{i+1} - a_i \leq \delta_i$ for all $i \in \overline{0, k-1}$, where $\delta_i$ is the smaller root of \eqref{eqn:delta_equality} with $l=a_i$.
\end{definition}

Note that Definition~\ref{def:epichain} and Definition~\ref{def:subchain} are consistent: by Proposition~\ref{stt:properties_of_delta_equation},~\eqref{eqn:delta_equality}~has real non-negative roots, and $1 - (a_i + \delta_i + m/\varepsilon) \geq 0$, which in case of subchain guarantees that for all~$i$, $a_i \leq 1 - m/\varepsilon$. Therefore, this condition is present in the definition of an epichain and omitted in the definition of a subchain.

\begin{Proposition}\label{stt:linear_subchain}
    For any $j,k,m,n$ and $\varepsilon$ satisfying conditions in a definition, the linear sequence $\{d_i\}_{i=j}^k$ defined by $d_i = d_j + i\delta_j$, where $\delta_j$ is the smaller root of \eqref{eqn:delta_equality} with $l = d_j$ and $i=j$, is a subchain.
\end{Proposition}

\begin{proof}
    We utilize the function $\Gamma_{m,\varepsilon}(\delta)$ introduced in~\eqref{eqn:gamma_def}. Transforming the right-hand side of~\eqref{eqn:gamma_delta*}, we have
    \[
    \Gamma_{m,\varepsilon}(\delta_i) = \frac{1 - l_j - m/\varepsilon - (i - j)\delta_j}{n - i} = \frac{n - j}{n - i}\Gamma_{m,\varepsilon}(\delta_j) - \frac{i - j}{n - i}\delta_j \overset{(*)}{\geq} \Gamma_{m,\varepsilon}(\delta_j)\,,
    \]
    $(*)$ follows from the fact that $\Gamma_{m,\varepsilon}(\delta) \geq \delta$. Since $\Gamma_{m, \varepsilon}(\delta)$ is a non-decreasing function, $\delta_j \leq \delta_i$, and $d_{i+1} - d_i \leq \delta_i$, which proves that $\{ d_i \}_{i = j}^k$ is a subchain.
\end{proof}

Lemma~\ref{lem:subchain_less_epichain} studies the relation between the set of epichains and the set of subchains which start with the same element in terms of element-wise sequence comparison (e.g., $\{a_i\}_{i=j}^k \geq \{ b_i \}_{i=j}^k$ means that $a_i \geq b_i$ for all $i$ in $\overline{0, k}$).

\begin{lemma}\label{lem:subchain_less_epichain}
    Let $\{l_i\}_{i=j}^k$ and $\{d_i\}_{i=j}^k$, where $l_j = d_j$, be an epichain and a subchain, respectively, for the same parameters $j,k,m,n, \varepsilon$. Then $\{ l_i \}_{i=j}^k \geq \{ d_i \}_{i=j}^k$.
\end{lemma}

\begin{proof}
Assume $\{d_i\}_{i=j}^k \nleq \{l_i\}_{i=j}^k$. Since $d_j = l_j$, there must be a first index $i$ such that $d_{i+1} > l_{i+1}$, while $d_i \leq l_i$. From the definitions of the sequences, we know $d_{i+1} \leq d_i + \delta_i^d$ and $l_{i+1} \geq l_i + \delta_i^l$. Combining these with our assumption gives the inequality $d_i + \delta_i^d > l_i + \delta_i^l$.

Transforming~\eqref{eqn:gamma_delta*}, we obtain
\begin{align*}
1 - \left(d_i + \delta_i^d + m/\varepsilon\right) &= (n - i)\Gamma_{m,\varepsilon}\left(\delta_i^d\right) - \delta_i^d\,, \\
1 - \left(l_i + \delta_i^l + m/\varepsilon\right) &= (n - i)\Gamma_{m,\varepsilon}\left(\delta_i^l\right) - \delta_i^l\,.
\end{align*}
The condition $d_i + \delta_i^d > l_i + \delta_i^l$ is then equivalent to 
\[
(n - i)\Gamma_{m,\varepsilon}\left(\delta_i^d\right) - \delta_i^d < (n - i)\Gamma_{m,\varepsilon}\left(\delta_i^l\right) - \delta_i^l\,,
\]
which implies $\delta_i^d < \delta_i^l$, since $i < n$ and $\Gamma_{m,\varepsilon}(\delta) - \delta$ is a non-decreasing function. Then, $d_{i+1} \leq d_i + \delta_i^d < l_i + \delta_i^l \leq l_{i+1}$, which contradicts our assumption.
\end{proof}

Now we can establish a lower bound on the final value of the barrier, $l_k$, and consequently an upper bound on $\norms{X_\cS^\dag}^2$. This bound, in turn, allows us to determine an optimal value for $\varepsilon$.

\begin{Proposition}\label{stt:orthonormal_bounds}
     Fix $X \in \R^{m \times n}$ ($XX^T=I$). Let $\varepsilon_{opt}$ be defined as
    \begin{equation}\label{eqn:optimal_epsilon}
    \begin{cases}
        \displaystyle\varepsilon_{opt} = \text{arbitrary number in}\ (0, 1)\,, &\text{if}\ m = 1\,, \\
        \displaystyle\varepsilon_{opt} =  n\frac{ 2 \left(\alpha - 1\right)+m \left(k \left(\alpha+m-2\right)-2 \alpha-m+3\right)}{ (k-1)m(k-m+1)}\,, &\text{if}\ m > 1\,,\\
    \end{cases} 
    \end{equation}
     where $\alpha = \sqrt{(k-1)m+1}$. Then, after $k$ iterations of running Algorithm~\ref{alg:one_greedy_iteration}-\ref{alg:l_epsilon_theory} (starting with $\cS = \varnothing$) with $\varepsilon = \varepsilon_{opt}$, $\norms{X_\cS^\dag}^2$ satisfies
    \begin{equation}\label{eqn:pinv_norm_bound}
    \begin{cases}
        \displaystyle \norms{X_\cS^\dag}^2 = \frac{1}{\lambda_m(Y_k)} \leq \frac{n}{k}\,, &\text{if}\ m = 1\,, \\
        \displaystyle \norms{X_\cS^\dag}^2 = \frac{1}{\lambda_m(Y_k)} \leq \frac{n}{m}\left(\frac{\alpha - 1}{\alpha - k}\right)^2\,, &\text{if}\ m > 1\,.\\
    \end{cases} 
    \end{equation}
\end{Proposition}

\begin{proof}
    Let us fix an arbitrary $\varepsilon \in \mathcal{I}_m$. The sequence of barrier values $\{l_i\}_{i=0}^k$ generated by Algorithm~\ref{alg:one_greedy_iteration}-\ref{alg:l_epsilon_theory} is an epichain, according to Proposition~\ref{stt:l_is_epichain}. Lemma~\ref{lem:subchain_less_epichain} establishes that this epichain is bounded below by any subchain that starts with the same initial value. Setting $d_0 = l_0$ in Proposition~\ref{stt:linear_subchain} provides exactly such a sequence: the linear subchain $\{d_i\}_{i=0}^k$. We can therefore state that $l_k \geq d_k$.
    
    The potential function property guarantees that after $k$ columns are selected, $\lambda_m(Y_k) \geq l_k + 1/\varepsilon$. Using the linear subchain bound, we have:
    \begin{equation}\label{eq:lambdam_cond}
    \lambda_m(Y_k) \geq d_k + \frac{1}{\varepsilon} = -\frac{m - 1}{\varepsilon} + k\delta_0 = \delta_0 \left( -\frac{m - 1}{\varepsilon \delta_0} + k \right)\,,
    \end{equation}
    where $\delta_0$ is the smaller root of the Equation~\ref{eqn:delta_equality} with $l = l_0 = -m/\varepsilon$ and $i = 0$. 

    To maximize this lower bound, we re-parametrize it. Let $\varepsilon' = \varepsilon \delta_0$. Equations~\ref{eqn:gamma_def} and~\ref{eqn:gamma_delta*} allow us to express $\delta_0$ as a function of $\varepsilon'$:
    \[
        \delta_0 = \frac{1}{n}\left(\frac{1 - \varepsilon'}{1 - \varepsilon' / m}\right)\,.
    \]
    Substituting it into~\eqref{eq:lambdam_cond} and maximizing over $\varepsilon'$ yields~\eqref{eqn:optimal_epsilon} and lower bound on $\lambda_m(Y_k)$, which is equivalent to upper bound in~\eqref{eqn:pinv_norm_bound}.
\end{proof}

\subsection{Heuristic approach for updating the barrier}\label{sec:l_epsilon_heuristic}

While the update strategy presented in Subsection~\ref{sec:final_l_estimation} is theoretically sound and sufficient to prove our main bound, its practical performance can be enhanced. The guaranteed barrier advancement, $\delta_i$, is a worst-case lower bound. In practice, after adding a column, it is often possible to advance the barrier much further than $\delta_i$, while keeping the potential on the same level. This creates a \enquote{performance surplus}, affording us the freedom to select the next state $(l_{i+1}, \varepsilon_{i+1})$ more strategically without compromising the bounds.

Our heuristic uses this surplus to adaptively control the algorithm's greediness by adjusting the position of the barrier $l$. A distant $l$ encourages a conservative selection that considers the global eigenvalue structure, which is ideal for early stages. A closer $l$ makes the selection aggressively prioritize the smallest eigenvalue, which is preferable in the final stages.

To implement the heuristic strategy safely, we introduce a \enquote{lookahead} function that estimates the guaranteed final performance from any intermediate state. Consider the state at the end of the iteration $i$ of running Algorithm~\ref{alg:one_greedy_iteration}. At this point, $i+1$ columns have been selected. We define the $B_{i+1}(l)$ as:
\begin{equation}\label{eqn:B_definition}
B_{i+1}(l) = l + (k - i - 1)\delta(l, \varepsilon) + \frac{1}{\varepsilon}\,,
\end{equation}
where $\varepsilon = \Phi_l(Y_{i+1})$, and $\delta(l, \varepsilon)$ is the smaller root of~\eqref{eqn:delta_equality} (with the iteration counter in that equation corresponding to the current state, i.e., $i+1$ columns selected). This function provides the guaranteed lower bound on the final value of $\lambda_m(Y_k)$ if, from this point forward, the algorithm were to proceed with the fixed-potential strategy from Subsection~\ref{sec:final_l_estimation}.

Our heuristic, presented in Algorithm~\ref{alg:l_epsilon_heuristic}, uses this function to guide its choices. The fundamental principle is to select a new state $(l_{i+1}, \varepsilon_{i+1})$ that maintains the initial performance guarantee, i.e., ensuring $B_{i+1}(l_{i+1}) \geq B_{0}$, where $B_{0} \equiv B_0(l_0)$ is the bound established at the start of the process. The algorithm identifies two candidate points: a conservative $l_{min}$ (the lowest barrier satisfying the guarantee) and an aggressive $l_{opt}$ (which maximizes $B_{i+1}(l)$). It then interpolates between these points, transitioning from the conservative to the aggressive choice as the selection progresses. A final safety check reverts to the theoretical method if the heuristic choice is found to be unsafe.

\begin{algorithm}[htbp]
\caption{Heuristic approach for updating $l$ and $\varepsilon$.}\label{alg:l_epsilon_heuristic}
\begin{algorithmic}[1]
\Require \parbox[t]{\dimexpr\linewidth-\algorithmicindent}{
Iteration number $i$, sampling parameter $k$,\\
matrix $Y_{i + 1} \in \R^{m \times m}$ with $m > 1$\footnote{The heuristic is designed for the case $m>1$. For $m=1$, Algorithm~\ref{alg:one_greedy_iteration} always picks the largest remaining element, regardless of the $l$ and $\varepsilon$.}, $l_i$ and $\varepsilon_i$.}
\Ensure New values of barrier and potential: $l_{i + 1}$, $\varepsilon_{i + 1}$.
\State Let $B_0$ be the theoretical lower bound on $\lambda_m(Y_k)$ from~\eqref{eqn:pinv_norm_bound}.
\State Using golden-section search, find candidate maximizer $l_{opt}$ of $B_{i+1}(l)$ on $[l_{-}, \lambda_m(Y_{i + 1}))$, where $l_{-} = -(m + 1)/(m - 1)$.
\State Using a bisection method, find the candidate $l_{min}$ by solving $B_{i+1}(l) = B_0$ on $[l_{-},l_{opt}]$.
\If{$i + 1 < k - m$}  \Comment{Conservative phase.}
\State $l_{trial} \gets l_{min}$.
\Else \Comment{Aggressive phase.}
\State $\lambda \gets (k - i - 2)/m$.
\State $\displaystyle l_{trial} \gets \lambda l_{min} + (1 - \lambda)l_{opt}$.
\EndIf

\If{$B_{i + 1}(l_{trial}) \geq B_0$} \Comment{Safety check.}
\State $l_{i + 1} = l_{trial}$, $\varepsilon_{i + 1} \gets \Phi_{l_{i + 1}}$.
\Else \Comment{Fallback.}
\State Use Algorithm~\ref{alg:l_epsilon_theory} to obtain $l_{i + 1}$ and $\varepsilon_{i + 1}$.
\EndIf

\State \textbf{return} $l_{i + 1}$, $\varepsilon_{i + 1}$.
\end{algorithmic}
\end{algorithm}

The numerical search for $l_{opt}$ and $l_{min}$ is performed on a bounded interval. For the lower bound, we choose $l_{-} = -(m+1)/(m-1)$, which can be shown to guarantee $B_{i+1}(l_{-}) \leq 0$. The derivation for this bound is as follows:
\begin{gather}
    \delta(l, \varepsilon) \leq \Gamma_{m,\varepsilon}(\delta) = \frac{1 - l - m/\varepsilon}{n - i - 1} \overset{(*)}{\leq} \frac{1}{n - i - 1}\,,\\
    \varepsilon = \sum_{j=1}^m \frac{1}{\lambda_j(Y) - l}\geq \frac{m}{1 - l} \implies \frac{1}{\varepsilon} \leq \frac{1 - l}{m}\,,
\end{gather}
$(*)$ follows from $\varepsilon = \sum_{j=1}^m (\lambda_j(Y) - l)^{-1} \leq - m/l$. This leads to an inequality 
\[
B_{i + 1}(l_{-}) \leq l_- + \frac{k - i - 1}{n - i - 1} + \frac{1 - l_-}{m} \leq \frac{(m - 1)l_- + (m + 1)}{m} = 0\,.
\]

Finally, we note that the values $l_{opt}$ and $l_{min}$ are termed \enquote{candidates} because we do not formally prove properties such as unimodality for $B_{i+1}(l)$, though it was consistently well-behaved in our experiments.

\begin{Proposition}\label{stt:heuristic_bounds}
    Let $X \in \R^{m \times n}$ ($XX^T =I$, $m > 1$). Consider the greedy selection process using the heuristic state update from Algorithm~\ref{alg:l_epsilon_heuristic} (Algorithm~\ref{alg:one_greedy_iteration}-\ref{alg:l_epsilon_heuristic}). If the process is initialized with $\varepsilon_0 = \varepsilon_{opt}$ from \eqref{eqn:optimal_epsilon}, the resulting submatrix $X_\cS$ satisfies the bounds stated in Proposition~\ref{stt:orthonormal_bounds}.
\end{Proposition}

\begin{proof}
If the heuristic update in Algorithm~\ref{alg:l_epsilon_heuristic} is never successfully applied, the process is identical to the theoretical one, and the statement follows directly from Proposition~\ref{stt:orthonormal_bounds}.

Otherwise, suppose $j - 1$ is the last iteration on which heuristic was applied, i.e. $l_{j}$ and $\varepsilon_{j}$  were chosen heuristically (for $m > 1$, $\mathcal{I}_m = (0, \infty)$, and thus $\varepsilon_{j} \in \mathcal{I}_m$) and all subsequent ones were not. The epichain and subchain analysis (Lemma~\ref{lem:subchain_less_epichain}) therefore applies to this final block of iterations. This guarantees that the final smallest eigenvalue is bounded by the performance function evaluated at step $j$:
\[
    \lambda_m(Y_k) \geq B_j(l_j)\,.
\]
The state $(l_j, \varepsilon_j)$ was chosen by the heuristic, so it must have passed the safety check of Algorithm~\ref{alg:l_epsilon_heuristic}. This check explicitly ensures $B_j(l_j) \geq B_0$. Combining these inequalities, $\lambda_m(Y_k) \geq B_j(l_j) \geq B_0$, which completes the proof.
\end{proof}

\subsection{The complete algorithm and main theorem}\label{sec:proof_main}

To generalize the algorithm, we relax the requirement that $X$ must have orthonormal rows. This can be achieved by performing an $\matdec{LQ}$ or singular value decomposition of $X$ and running the algorithm on the resulting matrix with orthonormal rows ($Q$ or $V^T$, respectively). This preprocessing step preserves the theoretical bound, as proven in Theorem~\ref{thm:main}.

The complete pseudocode of the algorithm is presented as Algorithm~\ref{alg:main}. This is a slightly more detailed version of Algorithm~\ref{alg:one_greedy_iteration}-\ref{alg:l_epsilon_heuristic} applied $k$ times, preceded by an $\matdec{LQ}$ decomposition of the input matrix to handle general matrices $X$.

\begin{algorithm}[htbp]
\caption{Deterministic greedy selection algorithm for subset selection.}\label{alg:main}
\begin{algorithmic}[1]
\Require $X \in \R^{m \times n}$ ($m \leq n$, $\rank X = m$), sampling parameter $k \in \overline{m, n}$.
\Ensure set $\cS \subseteq \overline{1,n} $ of cardinality $k$.
\State \textbf{initialize} $\displaystyle \begin{aligned}[t]
        & \cS \gets \varnothing,\ \cR \gets \overline{1, n},\ Y_0 \gets 0_{m\times m}, \\
        & \varepsilon_0 \gets \varepsilon_{opt},\ l_0 \gets -m/\varepsilon_0.
\end{aligned}$ \Comment{$\varepsilon_{opt}$ is defined in~\eqref{eqn:optimal_epsilon}.}
\State Compute thin $\matdec{LQ}$ decomposition of $X$, $X = LQ$. Assign $X \gets Q$.

\For{$i = 0,1,\dots,k-1$}
    \State $\delta_i \gets$ smaller root of~\eqref{eqn:delta_equality} using $l=l_i, \varepsilon=\varepsilon_i$.
    \State \parbox[t]{\dimexpr\linewidth-\algorithmicindent}{%
    Compute $\left(Y_i - (l_i + \delta_i)I\right)^{-1}$ using eigenvalue decomposition of $Y_i$. For each $j \in \cR$, use it to effectively compute $\Phi_{l_i + \delta_i}(Y_i + x_j x_j^T)$  via~\eqref{eqn:Phi_recalculation};
    }
    \[
    s \gets \argmin_{j \in \cR}\ \Phi_{l_i + \delta_i}(Y_i + x_j x_j^T)\,.
    \]
    \State $\cS \gets \cS \cup \{ s \},\ \cR \gets \cR \backslash \{s\}$, $\displaystyle Y_{i + 1} \gets Y_i + x_sx_s^T$.
    \State\label{itm:updating_l} \parbox[t]{\dimexpr\linewidth-\algorithmicindent}{%
    Compute eigenvalue decomposition\footnote{Instead of calculating the eigenvalue decomposition of $Y_{i + 1}$ from scratch, one can use faster rank-1 update~\cite{eigenproblem_recalculation}. In that case, we suggest supplementing State~\ref{itm:updating_l} with $X \gets U^TX,\ Y_{i + 1} \gets \Lambda$, where $Y_{i + 1} = U\Lambda U^T$ is the eigenvalue decomposition. Then $X \gets U^TX$ will be the only step of the Algorithm~with cubic complexity, as~\eqref{eqn:Phi_recalculation} will be calculated in $O (nm)$ for diagonal $Y_{i + 1}$, and eigenvalue decomposition of a rank 1 update of the diagonal $Y_i$ is calculated in $O(m^2)$.   
    } of $Y_{i + 1}$.}
    
    \State Apply Algorithm~\ref{alg:l_epsilon_heuristic} (or Algorithm~\ref{alg:l_epsilon_theory}, if $m = 1$) to obtain $l_{i + 1}$ and $\varepsilon_{i + 1}$.
\EndFor
\State \textbf{return} $\cS$
\end{algorithmic}
\end{algorithm}

\begin{theorem}\label{thm:main}
    There exists a deterministic algorithm (Algorithm~\ref{alg:main}) that, given a full-rank matrix $X \in \R^{m \times n}$ with $m < n$, and sampling parameter $k \in \overline{m, n}$, constructs a subset $\cS \subseteq \overline{1,n}$ of cardinality $k$. The algorithm ensures that $X_\cS$ has full rank, and
    \[
    \normxi{X_\cS^\dag}^2 \leq \frac{n}{m}\left(\frac{\sqrt{(k - 1)m + 1} - 1}{\sqrt{(k - 1)m + 1} - k}\right)^2 \normxi{X^\dag}^2\,,\quad \xi \in \{2,F\}\,.
    \]
    For $m = k = 1$ the bound should be understood in the limit $k \to 1+$, yielding $\normxi{X_S^\dag}^2 \leq n\normxi{X^\dag}^2,\ \xi \in \{2,F\}$. 
    
    Additionally, a slightly looser, but more interpretable bound is given by
    \[
    \normxi{X_\cS^\dag}^2 \leq \frac{n}{\left(\sqrt{k} - \sqrt{m - 1}\right)^2} \normxi{X^\dag}^2\,,\quad \xi \in \{2,F\}\,.
    \]
    
    The algorithm runs in $O(k m^3 + k m T_X)$ operations, where $T_X \geq n$ is the complexity of multiplying $X^T$ by a vector of length $m$. For a general dense matrix the total complexity becomes $O( n k m^2)$.
\end{theorem}

\begin{proof}
    \begin{enumerate}[wide]
        \item{Proof of bounds.} 
        To prove tighter bounds, we need to show that $\matdec{LQ}$ decomposition indeed allows us to generalize the bounds from Proposition~\ref{stt:heuristic_bounds} (Proposition~\ref{stt:orthonormal_bounds}, if $m = 1$). Suppose $X = LQ$ is the $\matdec{LQ}$ decomposition of $X$ and $\cS$ is some subset of column indices. Then, for $\xi \in \{2, F\}$
        \[
        \frac{\normxi{X_\cS^\dag}}{\normxi{X^\dag}} = \frac{\normxi{Q_\cS^\dag L^{-1}}}{\normxi{Q^\dag L^{-1}}} = \frac{\normxi{Q_\cS^\dag L^{-1}}}{\normxi{L^{-1}}} \leq \frac{\norms{Q_\cS^\dag} \normxi{L^{-1}}}{\normxi{L^{-1}}} = \norms{Q_\cS}\,,
        \]
        which proves the correctness of tighter bounds. 
        
        Now we prove the correctness of the looser bounds. The case $m=1$ is trivial. For $1 < m \leq k$, we start with the tight bound and factor out $\sqrt{km}$ from numerator and $\sqrt{k}$ from denominator:
        \[
         \frac{n}{m}\left(\frac{\sqrt{(k - 1)m + 1} - 1}{\sqrt{(k - 1)m + 1} - k}\right)^2 = 
            n \left(\frac{\sqrt{1 - \frac{m-1}{km}} - \frac{1}{\sqrt{km}}}{\sqrt{k} - \sqrt{m - 1}\sqrt{1 + \frac{k -m + 1}{k(m - 1)}}}\right)^2\,.
        \]
        Applying inequality  $\sqrt{1 + a} \leq 1 + a/2$ to $\sqrt{1 - \frac{m - 1}{km}}$ and $\sqrt{1 + \frac{k - m + 1}{k(m - 1)}}$ we obtain
        \begin{multline*}
            \frac{n}{m}\left(\frac{\sqrt{(k - 1)m + 1} - 1}{\sqrt{(k - 1)m + 1} - k}\right)^2 \leq 
            n \left(\frac{1 - \frac{m-1}{2km} - \frac{1}{\sqrt{km}}}{\sqrt{k} - \sqrt{m - 1} + \frac{k -m + 1}{2k\sqrt{m - 1}}}\right)^2 = \\
            = \frac{n}{\left(\sqrt{k} - \sqrt{m - 1}\right)^2} \left(\frac{1 - \frac{m - 1}{2km} - \frac{1}{\sqrt{km}}}{1 - \frac{\sqrt{k} + \sqrt{m - 1}}{2k\sqrt{m-1}}}\right)^2  \overset{(*)}{\leq} 
            \frac{n}{\left(\sqrt{k} - \sqrt{m - 1}\right)^2}\,,
        \end{multline*}
        $(*)$ follows from comparing numerator and denominator of the remaining fraction:
        \[
            \frac{m-1}{2km} + \frac{1}{\sqrt{km}} \geq \frac{\sqrt{k} + \sqrt{m - 1}}{2k\sqrt{m-1}} \Longleftrightarrow
            \sqrt{\frac{m - 1}{m}} \left( 2 - \frac{1}{\sqrt{km}} \right) \geq 1\,,
        \]
        where the left-hand side of the latter is monotonously increasing function of $m$ and $k$, and the inequality holds even in the worst case $k = m = 2$. 
        
        \item{Proof of asymptotic complexity.} The $\matdec{LQ}$ decomposition requires $O(nm^2)$ operations. The new matrix $X_{new}^T = Q^T$ after that still allows for fast multiplication by an arbitrary vector $v \in \R^m$, since $X_{new}^Tv = X_{old}^TL^{-T}v$, where the right-hand side requires $O(m^2 + T_X)$ operations.
        
        The $\matdec{LQ}$ decomposition is followed by $k$ iterations. On each iteration, we perform the following steps:
        \begin{enumerate}
            \item Calculate $\delta_i$ in $O(1)$ operations.
            \item Evaluate $\Phi_{l_i + \delta_i}(Y_i + x_j x_j^T)$ for $n - i$ columns. Using~\eqref{eqn:Phi_recalculation}, this can be done in $O(m^3 + mT_X)$ operations.
            \item Compute the eigenvalue decomposition of $Y_{i + 1}$, it requires $O \left(m^3 \right)$ operations.
            \item Update $l$ and $\varepsilon$. Since both $B_{i + 1}(l)$ and $\Phi_l(Y_{i + 1})$ can be computed in $O(m)$, applying bisection method or golden-section search on them is $O \left(m \right)$. Those algorithms are applied at most $3$ times, which makes the total cost of this step $O(m )$.
        \end{enumerate}
        Combining all mentioned steps, we obtain an overall asymptotic complexity of $O(km^3 + kmT_X)$ operations.\qedhere
    \end{enumerate}
\end{proof}

\subsection{Relation to previous studies}

Our approach is a direct refinement of the spectral sparsification framework established in~\cite{twice_ramanujan_sparcifiers} and adapted for subset selection in~\cite{faster_subset_selection}. These foundational methods utilize a dual-barrier structure to control both the smallest and largest eigenvalues. This process necessarily produces a set of non-binary column weights that must subsequently be converted into an unweighted selection.

Our key insight is that for the specific goal of minimizing the pseudoinverse norm, the upper barrier is unnecessary. We specialize this framework to a single barrier function. This simplification yields two significant advantages over the prior art. First, it enables direct, unweighted column selection, which is the foundation for our improved theoretical guarantees. Second, it provides the flexibility to develop a powerful adaptive update strategy for the barrier $l$, which is key to the algorithm's excellent practical performance.

\section{Numerical experiments}\label{sec:numerical_experiments}

We have implemented the subset selection algorithms and testing framework in C++ using the Eigen library for efficient matrix and vector operations, as well as numerical algorithms. For plotting and visualization, we utilize Matplotlib Python. The complete codebase, including examples and documentation, is openly available on GitHub as a compact header-only library \url{https://github.com/KozyrevIN/subset-selection-for-matrices}.

In our experiments, we compare the performance of Algorithm~\ref{alg:main} with that of other algorithms described in Subsubsection~\ref{sec:approximation_algorithms}. Short codenames of all compared methods are presented in Table~\ref{tbl:short_names}.

\begin{table}[htbp]
\centering
\begin{NiceTabular}{|l|l|}
\hline
Codename & Subset selection method \\
\hline \hline
spectral selection & Algorithm~\ref{alg:main} \\
\hline
spectral removal  & Algorithm~2 in~\cite{faster_subset_selection}  \\
\hline
dual set  & Algorithm~3 in~\cite{faster_subset_selection}  \\
\hline
random columns  & randomly selected $k$ columns \\
\hline
\end{NiceTabular}
\caption{Correspondence between algorithms and their codenames used on figures.}
\label{tbl:short_names}
\end{table}

Initially, we intended to include Algorithm~1 from~\cite{subset_selection_with_fixed_blocks} in our testing, but the straightforward implementation proved to be numerically unstable. The reasons for this instability are rooted in operations involving characteristic polynomials. Let $p_\cS(x)$ denote the characteristic polynomial of the matrix $X_\cS$ for a given set of column indices $\cS$ of cardinality less than $k$. The authors provide the following formula (Equation~22 in~\cite{subset_selection_with_fixed_blocks}) for the \enquote{expected} characteristic polynomial, whose smallest root is of interest:
\[
f_\cS(x) = \frac{(n - k)!}{\left(n - |\cS|\right)!}(x-1)^{-(n-m-k)}\partial_x^{k-|\cS|}(x-1)^{n-m-|\cS|}p_\cS(x)\,.
\]

While $f_\cS$ can be computed effectively in the polynomial basis $\{1,y,\dots,y^{n-|\cS|}\}$, where $y=x-1$, the resulting $f_\cS(y)$ has roots clustered near $-1$ when $k$ is small compared to $n$. This clustering renders the task of finding the smallest root extremely ill-conditioned~\cite{wilkinson1984perfidious}. Furthermore, simply reverting to the original variable $x = y + 1$ does not alleviate the issue, as it leads to catastrophic cancellations in the polynomial coefficients. These arguments are in good agreement with the experiment: we observed the emergence of negative and complex roots of $f_\cS(y)$ even for moderate values of $k$, $m$ and $n$ (e.g., $k = m = 5$, $n = 100$), while larger values of $k$ in otherwise identical setups yielded satisfactory performance. Stabilizing the algorithm remains an open question for future research.

\subsection{Experimental methodology}

The experiments were conducted on matrices of a fixed size $m=100$, $n=5000$. We varied the number of selected columns $k$ from $100$ to $5000$, generating $32$ random matrices for each value of $k$. The performance of the algorithms was evaluated using the metric $\norms{X^\dag}/\norms{X_\cS^\dag}$, where larger values correspond to a better result. The plots show the mean values of the metric, standard deviations, and theoretical guarantees.

It should be emphasized that all algorithms, as required in Problem~\ref{prm:subset_selection}, return submatrices of full rank. A random selection of columns, however, can lead to a singular submatrix. To demonstrate this, we use the following convention in the plots: if the submatrix $X_\cS$ is singular, then $\normxi{X_\cS^\dag}^2 = \infty$, which makes the metric value zero.

\subsection{Experiment 1: Matrices with orthonormal rows}

We use matrices with orthonormal rows, sampled from Circular Orthogonal Ensemble~\cite{9eafeb2573aa4d7a9d3f0f17ec8c9af5}. This scenario models one of the key applications of Problem~\ref{prm:subset_selection}~--- selecting rows/columns from a matrix of leading singular vectors to construct low-rank approximations or to select key features. 

\begin{figure}[htbp]
\centering
\includegraphics[width=\textwidth]{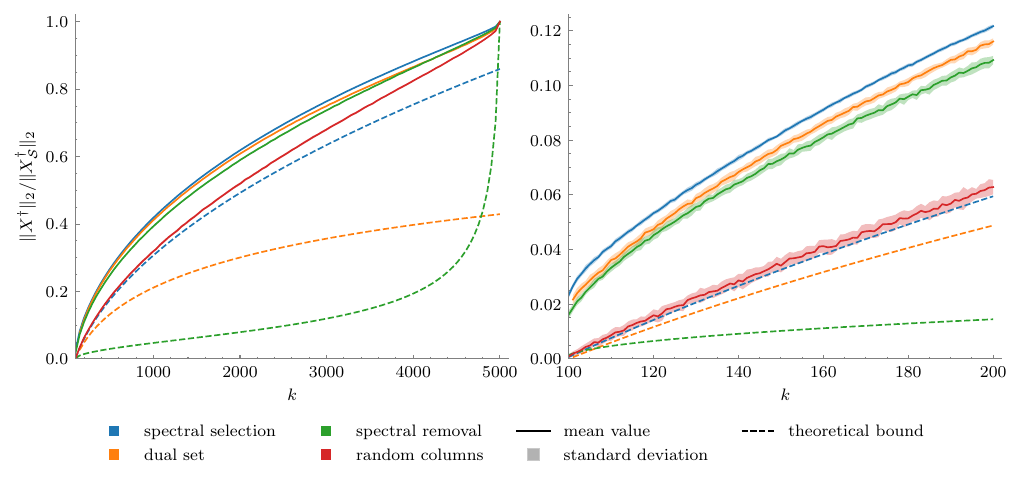}
\caption{Algorithm performance on matrices with orthonormal rows sampled from the Circular Orthogonal Ensemble ($m=100, n=5000$).}
\label{fig:tests_orthonormal}
\end{figure}

As shown in Figure~\ref{fig:tests_orthonormal}, our proposed algorithm (spectral selection) consistently outperforms the other deterministic methods across the entire range of $k$, with the gap between it and other algorithms especially perceptible for small values of $k$.

\subsection{Experiment 2: Incidence matrices of a random graph}

In this experiment, we test the algorithms on a problem related to graph theory: finding a spanning sub-graph with high algebraic connectivity~\cite{lamperskisimple}. The input matrices for this task are constructed from the singular vectors of a graph's incidence matrix. Specifically, we select columns from the matrix $V^T$, where $V$ contains the first $m$ right singular vectors of the oriented edge-vertex incidence matrix of a random weighted connected graph. 

To generate those matrices, we followed a four-step procedure: 
\begin{enumerate}
    \item Generated an unweighted graph with $m + 1$ vertices and $n$ edges from a uniform distribution.
    \item Verified the graph's connectivity and retried if necessary.
    \item Assigned a weight uniformly sampled from $(0, 1)$ to each edge.
    \item Performed a truncated singular value decomposition of the resulting edge-vertex incidence matrix to obtain an $m \times n$ matrix of its singular vectors.
\end{enumerate}

\begin{figure}[htbp]
\centering
\includegraphics[width=\textwidth]{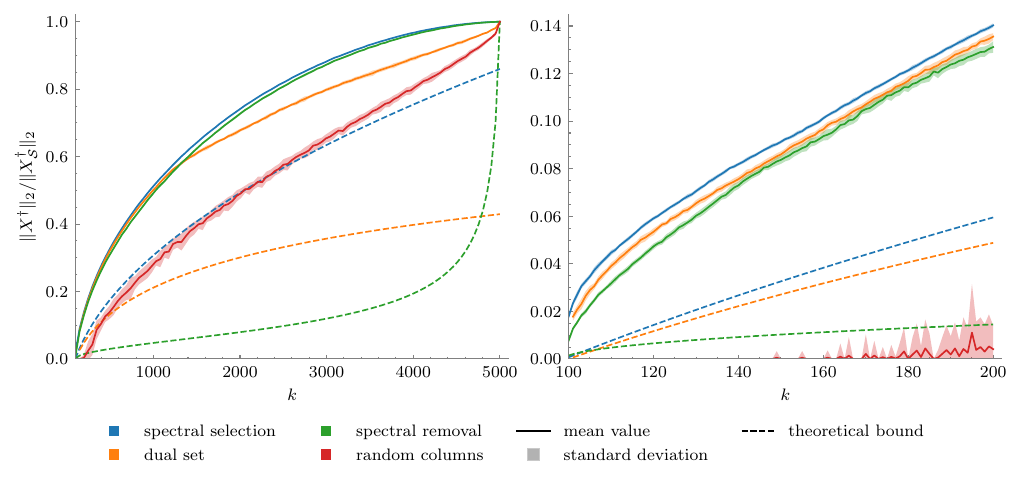}
\caption{Algorithm performance on incidence matrices of the random weighted connected graph ($m=100, n=5000$).}
\label{fig:tests_weighted_graph}
\end{figure}

The results presented in Figure~\ref{fig:tests_weighted_graph} corroborate the findings from the first experiment. Our algorithm again demonstrates superior practical performance, achieving the best metric value among all tested deterministic methods.

\section{Conclusion}

In this paper, we addressed the subset selection problem for matrices, focusing on the development of a deterministic greedy algorithm to select $k$ columns from a matrix such that the spectral norm of the resulting submatrix's pseudoinverse is minimized. Our proposed method, Algorithm~\ref{alg:main}, builds upon the spectral sparsification framework~\cite{twice_ramanujan_sparcifiers, column_based_reconstruction, faster_subset_selection} but introduces key modifications which allow us to tailor the selection process specifically for this objective while maintaining the same $O(nkm^2)$ asymptotic complexity.

The primary theoretical contribution of our work is a new, stronger bound on the resulting norm of the pseudoinverse, which is formally stated in Theorem~\ref{thm:main}. To our knowledge, the presented bound is the best available one for the spectral norm when $m + 3 < k \leq n/m - 1$. These improved guarantees have direct implications for other areas of numerical linear algebra. The accuracy of column-based $\matdec{CW}$ and $\matdec{CUR}$ low-rank matrix approximations is fundamentally linked to the solution of the subset selection problem; thus, our work directly translates to tighter accuracy bounds for these important techniques.

To validate these theoretical advances and facilitate further research, we developed a comprehensive C++ implementation of our algorithm and its key competitors within a robust testing framework, which is made publicly available. Our numerical experiments, conducted using this framework, confirm the practical effectiveness of our algorithm, showing that it consistently outperforms existing state-of-the-art deterministic methods.

In summary, our research provides a new, practically effective tool for subset selection for matrices that advances the state-of-the-art with stronger theoretical guarantees in key parameter regimes, while also contributing a valuable open-source implementation for future applications and comparative studies.

\section*{Acknowledgements}
The research was funded by the Russian Science Foundation (project No. 25-21-00159).

\bibliographystyle{abbrv}
\bibliography{refs}

\end{document}